\documentclass{birkjour}
\usepackage{amsmath}
\usepackage{amsthm}
\usepackage{amssymb}
\usepackage{amsmath}
\usepackage{amsthm}
\usepackage{amssymb}
\usepackage{fancyhdr}
\usepackage{enumerate}
\usepackage{amscd}
\usepackage[all]{xy}
\usepackage{graphicx}
\usepackage{color}
\usepackage{enumitem}

\theoremstyle{remark}

\newtheorem{example}{Example}

\theoremstyle{plain}
\newtheorem{THM}{Theorem}
\newtheorem{lemma}{Lemma}[section]
\newtheorem{proposition}[lemma]{Proposition}
\newtheorem{theorem}[lemma]{Theorem}
\newtheorem{defi}[lemma]{Definition}
\newtheorem*{remark}{Remark}

\numberwithin{equation}{section}

\theoremstyle{remark}
\newtheorem*{ack}{Acknowledgements}

\begin{document}

%
%
%
%
%
%
%
%
%
\title[On the validity of angular momentum]{On the validity of the definition of angular momentum in general relativity}
\author[P.-N. Chen]{Po-Ning Chen}
\address{Department of Mathematics, Columbia University, New York, NY 10027}
\email{pnchen@math.columbia.edu}
\author[L.-H. Huang]{Lan-Hsuan Huang}
\address{Department of Mathematics, University of Connecticut, Storrs, CT 06269}
\email{lan-hsuan.huang@uconn.edu}
\author[M.-T. Wang]{Mu-Tao Wang}
\address{Department of Mathematics, Columbia University, New York, NY 10027}
\email{mtwang@math.columbia.edu}
\author[S.-T. Yau]{Shing-Tung Yau}
\address{Department of Mathematics, Harvard University, Cambridge, MA 02138}
\email{yau@math.harvard.edu}
\begin{abstract}
We exam the validity of the definition of the ADM angular momentum without the parity assumption. Explicit examples of asymptotically flat hypersurfaces in the Minkowski spacetime with zero ADM energy-momentum vector and \emph{finite non-zero} angular momentum vector are presented. We also discuss the Beig-\'{O} Murchadha-Regge-Teitelboim center of mass and study analogous examples in the Schwarzschild spacetime. 

\end{abstract}

\maketitle

\section{Introduction}

After decades of study, the energy-momentum proposed by Arnowitt, Deser, and Misner \cite{ADM:1962} for asymptotically flat initial data sets has been well-accepted as a fundamental concept in general relativity. Schoen-Yau's theorem  \cite{Schoen-Yau-81b} (see also Witten \cite{Witten-81})  establishes the most important positivity property of the definition. In addition, the rigidity property that the mass is strictly positive unless the initial data set can be embedded into the Minkowski spacetime is also obtained. Bartnik also proves the ADM energy is a coordinate invariant quantity \cite{Bartnik-87}. Above-mentioned properties hold under rather general asymptotically flat decay assumptions at spatial infinity. 

There have been considerable efforts and interests to complete the definitions of total conserved quantities by supplementing with the angular momentum and center of mass. For example, see  \cite{Regge-Teitelboim:1974} for the angular momentum  and \cite{Beig-OMurchadha:1987, Huisken-Yau-96, Corvino-Schoen-06} for the center of mass. Those definitions using flux integrals have applications in, for instance, the gluing construction of \cite{Corvino-Schoen-06, Chrusciel-Delay-03}. However, they are more complex for at least the following two reasons:
 \begin{itemize}[leftmargin=2em]
 \item[(1)] The definition involves not only an asymptotically flat coordinate system but also the asymptotic Killing fields.
 \item[ (2)] The corresponding Killing fields are either boost fields or rotation fields, which are of higher order near spatial infinity in comparison to translating Killing field used in the definition of the ADM energy-momentum vector, so there is the issue of finiteness of the integrals. 
 \end{itemize}
 
The issue of finiteness has been addressed by Ashtekar-Hansen \cite{Ashtekar-Hansen:1978}, Regge-Teitelboim \cite{Regge-Teitelboim:1974}, Chru\'{s}ciel \cite{Chrusciel-87}, etc. Among them, Regge-Teitelboim proposed a parity condition on the asymptotically flat coordinate system. In particular, explicit divergent examples violating such a parity condition were constructed in \cite{Huang:2010-CQG, Cederbaum-Nerz:2013,Chan-Tam:2014}.

In this note, we address the question about the validity of the ADM angular momentum and the Beig-\'{O} Murchadha-Regge-Teitelboim (BORT) center of mass without assuming the parity condition. In particular, we provide explicit examples of spacelike hypersurfaces of finite angular momentum and center of mass in the Minkowski and Schwarzschild spacetimes. 

In the following, we recall the definition of asymptotically flat initial data sets and state the main results of this note.

An initial data set is a three-dimensional manifold $M$ equipped with a Riemannian metric $g$ and a symmetric $(0,2)$-tensor $k$. On an initial data set, one can define the mass density $\mu$ and the current density $J$  by
\begin{align*}
		\mu &= \frac{1}{2}\left( R_g - | k |_g^2 + (\mbox{tr}_g k)^2\right),\\
		J &= \mbox{div}_g k - d (\mbox{tr}_g k).
\end{align*}
\begin{defi} \label{def-ADM-mass}
Let $q>\frac{1}{2}$, $p>\frac{3}{2}$ and  $\epsilon >0$. The initial data set $(M, g, k)$ is asymptotically flat if for some compact subset $K\subset M$, $M\setminus K$ consists of a finite number of components $M_1$, $\ldots$ , $M_I$ such that each $M_i$  (end) is diffeomorphic to the complement of a compact set in $\mathbb R^3$. Under the diffeomorphisms, 
\[
	g_{ij}-\delta_{ij}=O_2(r^{-q}), \quad k_{ij}=O_1({r^{-p}})
\]
and 
\[
	\mu = O(r^{-3-\epsilon}), \quad |J| = O(r^{-3-\epsilon}).
\]
\end{defi}
The subscript in the big $O$ notations indicates the order of derivatives which have the corresponding decay rates. Namely, $f=O_1(r^{a})$ means $|f| =O(r^{a})$ and $|\nabla f| =O(r^{a-1})$. 

Note that it is necessary to assume $q>\frac{1}{2}$ and $p>\frac{3}{2}$ in order to prove positivity, rigidity, and coordinate invariance of the ADM definition of energy-momentum and mass by Schoen-Yau \cite{Schoen-Yau-81b}, Witten \cite{Witten-81}, and Chru{\'s}ciel \cite{Chrusciel-88} (see also, for example, \cite{Chrusciel:1986, Chrusciel:1987a, Bizon-Malec:1986}).

It is often convenient to consider the conjugate momentum 
\[
	\pi=k-(\mbox{tr}_g k)g.
\]
It contains the same information as $k$ because $k = \pi - \frac{1}{2} (\mbox{tr}_g \pi)g$. Then the Einstein constraint equations become
\begin{align*}
		  R_g - | \pi |_g^2 +\frac{1}{2} (\mbox{tr}_g \pi)^2 &= 2\mu,\\
		\mbox{div}_g \pi &= J.
\end{align*}
We will refer $(M, g, \pi)$ as an initial data set below.

\begin{THM} \label{theorem:finite-angular-momentum}
Let $(M, g, \pi)$ be an asymptotically flat initial data set. In addition, suppose $\pi$ satisfies
\[
	\pi=\bar{\pi} r^{-p}+\pi^{(-3)}r^{-3}+o_1(r^{-3})
\]
where $\frac{3}{2}<p<3$ and  $\bar{\pi}$  and $\pi^{(-3)}$ are $(0,2)$-tensors independent of $r$ on the unit sphere $S^2$. Assume further  that $|J| = O(r^{-4-\epsilon})$ for some $\epsilon > 0$. Then the ADM angular momentum is always finite. 
\end{THM}

\begin{THM} \label{th:angular-momentum-Minkowski}
There exist asymptotically flat spacelike hypersurfaces in the Minkowski spacetime with zero ADM energy-momentum, but the ADM angular momentum is finite and non-zero. 
\end{THM}

\begin{THM} \label{th:center-of-mass-Minkowski}
There exist asymptotically flat spacelike hypersurfaces in the Minkowski spacetime with zero ADM energy-momentum, but the  BORT center of mass is finite and non-zero.
\end{THM}

\begin{THM} \label{th:angular-momentum-Schwarzschild}
 There exist asymptotically flat spacelike hypersurfaces in the Schwarzschild spacetime of mass $m  > 0$ whose ADM energy-momentum vector is $(m, 0, 0, 0)$ and the ADM angular momentum is finite and greater than $m$. 
\end{THM}

The above examples do \emph{not} satisfy the Regge-Teitelboim condition. Hence, it is unclear whether the angular momentum and center of mass satisfy the corresponding change of coordinates when they are computed with respect to another asymptotically flat coordinate system~(\emph{cf.} \cite{Huang-09}).  The properties of the ADM and BORT definitions are in contrast to those of the recent definition of total conserved quantities on asymptotically flat initial data sets in \cite{Chen-Wang-Yau:2013}, where, for example, the new definitions of angular momentum and center of mass integrals always vanish for hypersurfaces in the Minkowski spacetime.

This note is organized as follows. In Section~\ref{sec:definitions}, we rewrite the ADM angular momentum integral in the spherical coordinates. In Section~\ref{sec:finite}, we  develop  criteria to ensure the finiteness of the ADM angular momentum and BORT center of mass and prove Theorem~\ref{theorem:finite-angular-momentum}.  In Section~\ref{sec:Minkowski}, we discuss examples of hypersurfaces in the Minkowski spacetime and prove Theorem~\ref{th:angular-momentum-Minkowski} and Theorem~\ref{th:center-of-mass-Minkowski}. In Section~\ref{sec:Schwarzschild}, we prove Theorem~\ref{th:angular-momentum-Schwarzschild}.

\section{ADM angular momentum and BORT center of mass} \label{sec:definitions}

We first recall the definition of ADM angular momentum and BORT center of mass.

\begin{defi}[\cite{Regge-Teitelboim:1974, Beig-OMurchadha:1987}] \label{de:com-am}
Let $(M, g, \pi)$ be an asymptotically flat initial data set. The center of mass $C$ and angular momentum $J(Y) $ with respect to a rotation vector field $Y=\frac{\partial}{\partial x^i} \times \vec{x}$, for some $i=1,2,3$, are defined by
\begin{align} \label{de:center-of-mass}
\begin{split}
	C^{i}
	= & \frac{1}{ 16 \pi } \lim_{r\rightarrow \infty} \int_{|x| = r}  \left[ x^i\sum_{j,k}\left(\frac{\partial g_{jk}}{\partial x^k}-\frac{\partial g_{kk} }{\partial x^j}\right)\frac{x^j}{|x|}\right. \\
	&\qquad \qquad \qquad \qquad \left. -\sum_k \left((g_{k i} -\delta_{ki})\frac{x^k}{|x|} -(g_{kk} - \delta_{kk})\frac{x^{i}}{|x|}\right) \right] \, d\sigma_0 
\end{split}
\end{align}
and 
\begin{align} \label{eq:angular-momentum-def}
	J(Y)= \frac{1}{ 8 \pi } \lim_{r\rightarrow \infty} \int_{|x| = r} \sum_{j,k} \pi_{jk} Y^j \frac{x^k}{|x|} \, d\sigma_0,
\end{align}
where $d \sigma_0$ is the area measure of the standard sphere of radius $r$. 
\end{defi}

Our goal is to express the ADM angular momentum in the spherical coordinates. Let  $\{r, u^a\},  a = 1, 2,$ be the spherical coordinates corresponding to $\{x^i\}$.
 The rule for change of variable is $x^i=r\tilde{x}^i$ where $\tilde{x}^i, i=1, 2, 3,$ are the three coordinate functions on the unit sphere $S^2$. We have
\[g_{ij} dx^i dx^j=(g_{ij} \tilde{x}^i \tilde{x}^j) dr^2+
	2 r g_{ij} \frac{\partial \tilde{x}^i}{\partial u^a} \tilde{x}^j du^a dr+
	r^2g_{ij} \frac{\partial \tilde{x}^i}{\partial u^a}\frac{\partial \tilde{x}^j}{\partial u^b} du^a du^b.
\]
Similarly,
\begin{align*}
	\pi_{ij} &= (\pi_{ij} \tilde{x}^i \tilde{x}^j) dr^2+
	2 r \pi_{ij} \frac{\partial \tilde{x}^i}{\partial u^a} \tilde{x}^j du^a dr+
	r^2\pi_{ij} \frac{\partial \tilde{x}^i}{\partial u^a}\frac{\partial \tilde{x}^j}{\partial u^b} du^a du^b\\
	& = \pi_{rr} dr^2 + 2 \pi_{ra} dr du^a + \pi_{ab} du^a du^b.
\end{align*}
In particular, we see that if $\pi_{ij}=O(r^{-p})$, then $\pi_{rr}=O(r^{-p})$, $\pi_{ra}=O(r^{-p+1})$, and $\pi_{ab}=
O(r^{-p+2})$. 

Let $\tilde \sigma$ and $\tilde \epsilon$ be the standard metric  and area form on the unit sphere, respectively. 
\begin{proposition} \label{pr:angular-momentum-spherical}
Let $(M, g, \pi)$ be asymptotically flat and let $\{ x^i\}$ be an asymptotically flat coordinate system. Let  $\{r, u^a\}, a = 1, 2,$ be the spherical coordinates corresponding to $\{x^i\}$. Then the angular momentum integral \eqref{eq:angular-momentum-def} can be written in the spherical coordinates:
\begin{align*}
	J (x^i \frac{\partial}{\partial x^j}-x^j\frac{\partial}{\partial x^i}) =-\frac{1}{8\pi}\lim_{r\to \infty} r^2 \int_{S^2} \epsilon^{ij}_{\,\,\,l} \tilde{x}^l\tilde{\epsilon}^{bc} \partial_b\pi_{rc} \, d\mu_{S^2}.
\end{align*}
where $d\mu_{S^2}$ is the standard measure on the unit sphere.
\end{proposition}

It is understood that $\pi_{rc}=\pi_{rc}(r, u^a)$ is considered as a one-form on $S^2$ that depends on $r$ and the integral  $\int_{S^2} \epsilon^{ij}_{\,\,\,l} \tilde{x}^l\tilde{\epsilon}^{bc} \partial_b\pi_{rc} \, d\mu_{S^2}$ becomes a function of $r$ only.
Note that $\tilde{\epsilon}^{bc} \partial_b\pi_{rc} = *d \pi_{rc}$ where $*$ is the Hodge dual with respect ot the metric $\tilde \sigma$. In particular, this is zero when $\pi_{rc}(r, u^a)$ is a closed one-form on $S^2$. This fact will be used in Proposition \ref{prop:angular-momentum-sphere}.
\begin{proof}
To compute the angular momentum, we need the following change of variable formulae:
\begin{align*}
	&\frac{\partial}{\partial x^i}=\frac{1}{r}\frac{\partial \tilde{x}^i}{\partial u^a}\tilde{\sigma}^{ab} \frac{\partial}{\partial u^b}+\tilde{x}^i \frac{\partial}{\partial r},\\
	&x^i \frac{\partial}{\partial x^j}-x^j\frac{\partial}{\partial x^i}= (\tilde{x}^i \frac{\partial \tilde{x}^j}{\partial u^a}
-\tilde{x}^j\frac{\partial \tilde{x}^i}{\partial u^a})\tilde{\sigma}^{ab}\frac{\partial}{\partial u^b}.
\end{align*}
Hence, we obtain
\[
	\pi(x^i \frac{\partial}{\partial x^j}-x^j\frac{\partial}{\partial x^i},\frac{\partial}{\partial r})=(\tilde{x}^i \frac{\partial \tilde{x}^j}{\partial u^a}
-\tilde{x}^j\frac{\partial \tilde{x}^i}{\partial u^a})\tilde{\sigma}^{ab}\pi_{br}.
\]
It is easy to check that
\[
	\tilde{x}^i \frac{\partial \tilde{x}^j}{\partial u^a}
-\tilde{x}^j\frac{\partial \tilde{x}^i}{\partial u^a}=\tilde{\epsilon}_a^{\,\,b} \epsilon^{ij}_{\,\,\,\,l}\partial_b \tilde{x}^l,
\] 
where $\tilde{\epsilon}_a^{\,\,b}$ is the area form on $S^2$ and $\epsilon^{ij}_{\,\,\,\,l}$ is the volume form on $\mathbb{R}^3$, raised by the standard metrics on $S^2$ and $\mathbb{R}^3$, respectively. Therefore, 
\begin{align*}
	\int_{S^2} (\tilde{x}^i \frac{\partial \tilde{x}^j}{\partial u^a}
-\tilde{x}^j\frac{\partial \tilde{x}^i}{\partial u^a})\tilde{\sigma}^{ac}\pi_{rc} d\mu_{S^2}=-\int_{S^2} \epsilon^{ij}_{\,\,\,l} \tilde{x}^l\tilde{\epsilon}^{bc} \partial_b\pi_{rc} d\mu_{S^2}.
\end{align*}
\end{proof}

\section{Finite angular momentum} \label{sec:finite}
Let $(M, g, \pi)$ be an initial data set where $\pi = k - (\mbox{tr}_g k)g$ is the conjugate momentum. In this section, we consider some criteria on $(g,\pi)$ to ensure the finiteness of the ADM angular momentum integral.

\subsection{Leading terms of the momentum tensor}

Assume that the momentum tensor $\pi$ has the following expansion
\begin{align} \label{eq:momentum-tensor}
	\pi=\bar{\pi} r^{-p}+\pi^{(-3)}r^{-3}+o_1(r^{-3})
\end{align}
where $\bar{\pi}$ and $\pi^{(-3)}$ are symmetric $(0,2)$-tensors independent of $r$ on $S^2$ and $\frac{3}{2}<p<3$. This condition is closely related to the Ashtekar-Hansen condition~\cite{Ashtekar-Hansen:1978}. We can rewrite $\pi$ in \eqref{eq:momentum-tensor} in the spherical coordinates, 
\begin{align} \label{eq:momentum-tensor-sphere}
	\pi = r^{-p} \beta dr^2 + 2r^{1-p} \alpha_a dr du^a +r^{2-p} h_{ab} du^a du^b + O_1(r^{-3}),
\end{align}
for a function $\beta$, a one-form $\alpha_a$, and a symmetric tensor $h_{ab}$ on $S^2$ independent of $r$. We will show that if $(M,g,\pi)$ is asymptotically flat and $\pi$ satisfies \eqref{eq:momentum-tensor}, then the angular momentum is always finite.

We recall the following lemma.
\begin{lemma}[{\cite[Lemma 2.3]{Huang-Schoen-Wang:2011}}] \label{lemma:divergence}
 Let $h$ be a symmetric $(0,2)$ tensor on $\mathbb{R}^3$ satisfying 
\[h = h_{00} dr^2 + 2 h_{0a} dr du^a + h_{ab} du^a du^b.\]
Then
\begin{align*}
	\textup{div}_{\delta} h &= \left[  r^{-2} \frac{\partial }{\partial r} (r^2 h_{00}) + r^{-2} \frac{1}{\sqrt{\tilde{\sigma}}} \frac{\partial }{\partial u^a} (\sqrt{\tilde{\sigma}}\tilde{\sigma}^{ab} h_{0a}) - r^{-3} \tilde{\sigma}^{ab} h_{ab}\right] dr\\
	&\quad +r^{-2} \left[  \frac{\partial}{\partial r} (r^2 h_{0a}) + \frac{1}{\sqrt{\tilde{\sigma}}} \frac{\partial}{\partial u^a} (\sqrt{\tilde{\sigma}}\tilde{\sigma}^{bc} h_{ac}) + \frac{1}{2} h_{bc} \frac{\partial }{\partial u^a} \tilde{\sigma}^{bc} \right] du^a.
\end{align*}
\end{lemma}
\begin{proposition}
Let $(M, g, \pi)$ be asymptotically flat. In addition, suppose $\pi$ satisfies \eqref{eq:momentum-tensor-sphere}. Then 
\begin{align*}
		\textup{div}_g \pi &=  r^{-1-p} \left[ (2-p)\beta + \tilde{\nabla}^a \alpha_a - h\right] dr\\
		&\quad + r^{-p} \left[ (3-p) \alpha_a + \tilde{\nabla }_b \hat{h}^b_a + \frac{1} {2} \frac{\partial}{\partial u^a} h\right] du^a + O(r^{-1- p-q}).
\end{align*}
	In particular, if  $|J|=O(r^{-4-\epsilon})$ for some $\epsilon > 0$, then the constraint equations imply
\begin{align}
	& (2-p)\beta + \tilde{\nabla}^a \alpha_a - h = 0 \notag \\ 
	& (3-p) \alpha_a + \tilde{\nabla }_b \hat{h}^b_a + \frac{1} {2} \frac{\partial}{\partial u^a} h= 0\label{eq:momentum-constraint-2},
\end{align}
where $h = \tilde{\sigma}^{ab} h_{ab}$ and $\hat{h}_{ab} = h_{ab} - \frac{1}{2} h \tilde{\sigma}_{ab}$.
\end{proposition}
\begin{proof}
The constraint equations imply $\textup{div}_g \pi = O(r^{-4-\epsilon})$. By the assumption $g = \delta + O(r^{-q})$, we have $\textup{div}_g \pi  = \textup{div}_{\delta} \pi + O(r^{-1-q-p})$. The proposition follows from applying Lemma~\ref{lemma:divergence}.
\end{proof}

\begin{proof}[Proof of Theorem~\ref{theorem:finite-angular-momentum}]
By \eqref{eq:momentum-constraint-2} and observing $\tilde{\nabla}_c \partial_a h= \tilde{\nabla}_a \partial_c h$, we obtain
\[
	\tilde{\epsilon}^{ac} \tilde{\nabla}_c \alpha_a =  -\frac{1}{3-p} \tilde{\epsilon}^{ac} \tilde{\nabla}_c \tilde{\nabla}_b \hat{h}^b_a.
\]
This is always perpendicular to $\tilde{x}^l$ by integration by parts twice and the equation $\tilde{\nabla}_b \tilde{\nabla}_c \tilde{x}^l = -\tilde{x}^l \tilde{\sigma}_{bc}$. Hence, by $\pi_{ra}=r^{1-p}\alpha_a+r^{-2} \pi_{ra}^{(-2)}+o_{1}(r^{-2})$ and Proposition~\ref{pr:angular-momentum-spherical},
\begin{align*}
&J (x^i \frac{\partial}{\partial x^j}-x^j\frac{\partial}{\partial x^i})\\
 = &-\frac{1}{8\pi}\lim_{r \to \infty}\left[ r^{3-p}  \int_{S^2} \epsilon^{ij}_{\,\,\,\,l}\tilde{x}^l \tilde{\epsilon}^{ac} \tilde{\nabla}_c \alpha_a \right] d \mu_{S^2}
-\frac{1}{8\pi}  \int_{S^2} \epsilon^{ij}_{\,\,\,\,l}\tilde{x}^l \tilde{\epsilon}^{ac} \tilde{\nabla}_c \pi_{ra}^{(-2)} d \mu_{S^2}\\
= &-\frac{1}{8\pi} \int_{S^2} \epsilon^{ij}_{\,\,\,\,l}\tilde{x}^l \tilde{\epsilon}^{ac} \tilde{\nabla}_c \pi_{ra}^{(-2)}\, d \mu_{S^2}.
\end{align*}
for all $i, j \in \{ 1, 2, 3\}$.
\end{proof}

We remark that Theorem~\ref{theorem:finite-angular-momentum} can be compared with the following example of divergent angular momentum. Although the leading term of $\pi$ in Example~\ref{ex:divergent-angular} is of the form $\bar{\pi} r^{-2}$ for a symmetric $(0,2)$-tensor $\bar{\pi}$ on $S^2$,  it does not contradict Theorem~\ref{theorem:finite-angular-momentum} since the next term is of the order $r^{-2-q}$ where $q$ is strictly less than $1$. In fact,  this next term cannot be of the order $r^{-3}$, in view of Theorem~\ref{theorem:finite-angular-momentum}.


\begin{example} [{\cite[Section 3]{Huang:2010-CQG}}] \label{ex:divergent-angular}
Given $q \in (\frac{1}{2}, 1)$ and functions $\alpha, \beta$ defined on the unit sphere, there exists a vacuum initial data set $(\mathbb{R}^3, g, \pi)$ with the following expansions at infinity
\begin{align*}
	g_{ij}dx^i dx^j &= \left( 1 + \frac{A}{r} + \frac{\alpha}{2r} \right) dr^2 + \left( 1 + \frac{A}{r} - \frac{\alpha}{2r} \right)r^2 \tilde{\sigma}_{ab} du^a du^b + O_2(r^{-1-q})\\
	\pi_{ij} dx^i dx^j& = \frac{\beta}{r^2} dr^2 + \frac{2}{r} \frac{\partial}{\partial u^a} \sum_i B_i \tilde{x}^i dr du^a + \sum_i B_i \tilde{x}^i  \tilde{\sigma}_{ab} du^a du^b + O_1(r^{-2-q}),
\end{align*}
for some constants $A, B_i, i = 1, 2, 3$. If one chooses 
\[
	\alpha =( \tilde{x}^1)^2, \quad \mbox{and} \quad \beta = \tilde{x}^1 \tilde{x}^3,
\]
then the angular momentum $J$ with respect to the rotation vector field $x^1\partial_3 - x^3 \partial_1$ diverges. 
\end{example}

\subsection{Fall-off rates of the initial data set}
We recall a theorem about finiteness and well-definedness of angular momentum for asymptotically flat manifolds by Chru{\'s}ciel \cite{Chrusciel-87}.
\begin{theorem}[\cite{Chrusciel-87}] \label{Chrusciel-87}
Let $(M, g, \pi)$ be an asymptotically flat initial data set. In addition, suppose
\[
	g = \delta+ O_2(r^{-q}), \quad \pi = O_1(r^{-p}), \quad \mbox{and} \quad |J| = O(r^{-4-\epsilon}).
\]
where $p+q>3$ and $\epsilon >0$. Then the ADM angular momentum is finite. 
\end{theorem}
\begin{remark}
The theorem is proved by applying the divergence theorem to \eqref{eq:angular-momentum-def} and then observing
\[  \textup{div}_g(\pi_{jk} Y^j) =  (\nabla^i \pi_{ij}) Y^j + \pi^{ij} (\mathfrak{L}_Y g)_{ij}. \]
The constraint equation implies that $ \nabla^i \pi_{ij} =J_j $ and 
 the fall-off rate implies that $\pi^{ij} (\mathfrak{L}_Y g)_{ij}$ is integrable. Hence, $\textup{div}_g(\pi_{jk} Y^j)$ is integrable.  Therefore, the angular momentum \eqref{eq:angular-momentum-def}  is finite and is independent of the family of surfaces used to compute the limit. 
\end{remark}

\section{Hypersurfaces in the Minkowski spacetime} \label{sec:Minkowski}
A spacelike slice in the Minkowski spacetime can always be written as the graph $t = f(r, u^a)$ for some function $f$ defined on $\mathbb{R}^3$, where $\{ r, u^a\}, a = 1, 2,$ are the spherical coordinates. \subsection{Non-zero angular momentum}

\begin{theorem}\label{th:hypersurface-angular-momentum}
Suppose $f = r^{\frac{1}{3}} A(u^a) $ where $A = \tilde{x}^1(\tilde{x}^2)^3$. Then the hypersurface $t = f$ in the Minkowski spacetime satisfies
\[	
	E= 0, \quad |P| = 0, \quad \mbox{and} \quad |C|=0. 
\] 
The ADM angular momentum $J(Y)$ is finite for any $Y = \frac{\partial}{\partial x^i} \times \vec{x}, i = 1, 2, 3$ and  
\[
	J (x^1 \frac{\partial}{\partial x^2}-x^2\frac{\partial}{\partial x^1}) = \frac{2}{3\cdot7\cdot 11}.
\]
\end{theorem}

\begin{remark}
This example is exactly on the borderline case of Theorem~\ref{Chrusciel-87} with 
\[
	g = \delta+ O_2(r^{-4/3})\quad and \quad \pi = O_1(r^{-5/3}) 
\]
and $p+q=3$.
\end{remark}

To prove the above theorem, we need the following computational results. Denote by $\bar{g}=dr^2+r^2\tilde{\sigma}_{ab} du^a du^b$ the standard metric on $\mathbb{R}^3$. The induced metric on the hypersurface $t= f$ is 
\[
	{g}=(1-f_r^2)dr^2-2f_rf_a dr du^a+(r^2\tilde{\sigma}_{ab}-f_a f_b) du^a d u^b
\]
and the second fundamental form of the hypersurface is 
\[
	k = \frac{1}{\sqrt{1-|\bar{\nabla} f|^2}}(\bar{\nabla}_i\bar{\nabla}_j f )dx^i dx^j,
\] 
where $\bar{\nabla}$ is covariant derivative of $\bar{g}$ and $\{x^i\}$ is an arbitrary coordinate chart on $\mathbb{R}^3$. In the spherical coordinate system of $\bar{g}$, the only non-trivial
 Christoffel symbols are $\bar{\Gamma}_{ab}^r=-r\tilde{\sigma}_{ab}$, $\bar{\Gamma}_{br}^a=r^{-1}\delta_b^a$, and $\bar{\Gamma}_{ab}^c$. 
The second fundamental form in the spherical coordinates is thus 
\[
	k=\frac{1}{\sqrt{1-|\bar{\nabla} f|^2}}\left[f_{rr} dr^2+2(f_{ra}-\bar{\Gamma}_{ar}^b f_b)drdu^a+(f_{ab}-\bar{\Gamma}_{ab}^r f_r-\bar{\Gamma}_{ab}^c f_c) du^a du^b\right].
\]

 \begin{proposition}
Assume $t=f(r, u^a)$ is a spacelike hypersurface in the Minkowski spacetime. Suppose the momentum tensor in the spherical coordinates is of the form $\pi= \pi_{rr} dr^2 + 2 \pi_{ra } dr du^a + \pi_{ab} du^a du^b$. Then
\begin{align} \label{eq:cross-term}
	\begin{split}
	\pi_{ra}=\frac{1}{\sqrt{1-|\bar{\nabla} f|^2}}\left[ f_{ra}-r^{-1} f_a+(f_{rr}+r^{-2}\tilde{\Delta} f+2r^{-1} f_r) f_r f_a\right. \\
\left.+\frac{f_r f_a}{{1-|\bar{\nabla} f|^2}}(\bar{g}^{kp}\bar{g}^{lq} f_p f_q)(\bar{\nabla}_k\bar{\nabla}_l f)\right],
	\end{split}
\end{align}
where $\tilde{\Delta} f$ denotes the Laplacian of $f(u^a, r)$ with respect to the standard metric $\tilde{\sigma}_{ab}$ on $S^2$ (only derivatives with respect to $u^a$ are involved).
\end{proposition}
 \begin{proof}
 We recall the relation between the Hessians of $f$ with respect to $g$ and $\bar{g}$:
 \begin{equation}\label{two.hessians} 
 	\nabla_i\nabla_j f=\frac{1}{1-|\bar{\nabla} f|^2}(\bar{\nabla}_i\bar{\nabla}_j f).
\end{equation}
We note that $k=\sqrt{1-|\bar{\nabla} f|^2 }(\nabla_i \nabla_j f) dx^i dx^j$ where $\nabla$ is the covariant derivative with respect to the induced metric $g$. Thus 
\[	
	\mbox{tr}_g k= \sqrt{1-|\bar{\nabla} f|^2 }\Delta f,
\] 
where $\Delta$ is the Laplace operator of $g$. Therefore, 
\begin{align*}
	\pi_{ra}&=k_{ra}-(\mbox{tr}_g k) g_{ra}\\
			&=\frac{1}{\sqrt{1-|\bar{\nabla} f|^2}}(f_{ra}-r^{-1} f_a)+ \sqrt{1-|\bar{\nabla} f|^2 } (\Delta f) f_r f_a.
\end{align*}
From \eqref{two.hessians}, we compute
\[
	\Delta f=\frac{1}{1-|\bar{\nabla} f|^2}(g^{kl}\bar{\nabla}_k\bar{\nabla}_l f)=\frac{1}{1-|\bar{\nabla} f|^2}\left(\bar{\Delta} f+\frac{1}{{1-|\bar{\nabla} f|^2}}(\bar{g}^{kp}\bar{g}^{lq} f_p f_q)(\bar{\nabla}_k\bar{\nabla}_l f)\right),
\] 
where we use
\[
	g^{kl}=\bar{g}^{kl}+\frac{\bar{g}^{kp}\bar{g}^{lq} f_p f_q}{1-|\bar{\nabla} f|^2}.
\]
Note that $\bar{\Delta}f=f_{rr}+r^{-2} (\tilde{\Delta} f)+2r^{-1} f_r$. Thus we obtain the desired identity.
\end{proof}

\begin{proposition} 
Let $f = A(u^a) r^p$ for some real number $p$ and a function $A(u^a)$ defined on $S^2$. If $0< p < \frac{1}{2}$, we have
\begin{align} \label{eq:cross-term-f}
\begin{split}
	\pi_{ra}&=A_a(p-1) r^{p-1}\\
	&\quad+ A_a\left[\frac{1}{2} p^2(3p+1) A^2+pA\tilde{\Delta}A+\frac{(p-1)}{2}|\tilde{\nabla} A|^2\right] r^{3p-3}+o_1(r^{3p-3}).
\end{split}
\end{align}
If $p=0$, we obtain
\begin{align} \label{eq:cross-term-p-zero}
	\pi_{ra} = - r^{-1} A_a - \frac{1}{2} r^{-3} |\tilde{\nabla} A|^2  A_a+ o_1(r^{-3}).
\end{align}
\end{proposition}
\begin{proof}
Note that 
\begin{align*}
	& f_r=pA r^{p-1},  \quad	f_{rr}=A p(p-1) r^{p-2}, \\
	& f_{ra}=A_a pr^{p-1}, \quad f_a=A_a r^p,	\quad \tilde{\Delta} f=(\tilde{\Delta} A)r^p.
\end{align*}
Thus,
\[
	|\bar{\nabla} f|^2=(f_r)^2+r^{-2}\tilde{\sigma}^{ab} f_a f_b=(p^2 A^2+\tilde{\sigma}^{ab} A_a A_b)r^{2p-2}.
\]
We check that the denominator of \eqref{eq:cross-term} has the following expansion
\begin{equation}
	\frac{1}{\sqrt{1-|\bar{\nabla} f|^2}}=1+\frac{(p^2 A^2+|\tilde{\nabla}A|^2)}{2}r^{2p-2}+o_1 (r^{2p-2}).
\end{equation}
Continuing the calculation of terms in $\pi_{ra}$ of \eqref{eq:cross-term}, we arrive at  
\begin{align*}
 &f_{ra}-r^{-1} f_a+(f_{rr}+r^{-2}\tilde{\Delta} f+2r^{-1} f_r) f_r f_a\\
&= A_a(p-1) r^{p-1}+\left((p(p-1) A+\tilde{\Delta}A) r^{p-2}+2 p Ar^{p-2}\right)(pA r^{p-1})(A_a r^p)\\
&=A_a(p-1) r^{p-1}
+( p(p+1) A+\tilde{\Delta}A)p A A_a  r^{3p-3}.
\end{align*}
Combining the above computations, we obtain the desired statement. 
\end{proof}

\begin{proposition} \label{prop:angular-momentum-sphere}
Let $t= r^{p} A(u^a)$ be a hypersurface in the Minkowski spacetime. Suppose $p=\frac{1}{3}$. Then the angular momentum integral is finite and 
\begin{align} \label{eq:angular-momentum-sphere}
	J (x^i \frac{\partial}{\partial x^j}-x^j\frac{\partial}{\partial x^i}) =-\frac{1}{24 \pi} \int_{S^2}  \epsilon^{ij}_{\,\,\,\,l}\tilde{x}^l * d\left( dA \left[A\tilde{\Delta}A-|\tilde{\nabla} A|^2\right]\right) d\mu_{S^2}.
\end{align}
\end{proposition}
\begin{proof}
By Proposition~\ref{pr:angular-momentum-spherical}, the terms in \eqref{eq:cross-term-f} which are closed one-forms on $S^2$ do not contribute to the angular momentum integral. By setting $p = \frac{1}{3}$ and letting $r$ go to infinity, we prove the identity. 
\end{proof}

Our goal is to find a function $A$ on $S^2$ so that \eqref{eq:angular-momentum-sphere} is not zero for some $i, j \in \{ 1, 2, 3\}$. 

\begin{lemma}\label{integral_lemma} 
Let $i, j \in \{ 1, 2, 3 \}$, $i\not= j$. Let $p, q$ be positive even integers. We have the  following inductive formulae
\begin{align*}
	\int_{S^2} (\tilde{x}^i)^p (\tilde{x}^j)^q \, d\mu_{S^2}=& \frac{p(p-1)}{(p+q)(p+q+1)}\int_{S^2} (\tilde{x}^i)^{p-2} (\tilde{x}^j)^q\, d\mu_{S^2} \\
                           &+\frac{q(q-1)}{(p+q)(p+q+1)}\int_{S^2}( \tilde{x}^i)^p (\tilde{x}^j)^{q-2}\, d\mu_{S^2},\\
\int_{S^2} (\tilde x^i)^{q} d\mu_{S^2} = & \frac{4 \pi}{1+q}, \\
\int_{S^2} (\tilde{x}^i)^2 (\tilde{x}^j)^q\, d\mu_{S^2}=&\frac{1}{(q+3)(q+1)} 4\pi,\\
	\int_{S^2} (\tilde{x}^i)^p (\tilde{x}^j)^q \, d\mu_{S^2}=& \frac{(p-1)(p-3)\cdots 3}{(q+1)(q+3)\cdots (q+p-3)}\frac{4\pi}{(q+p-1)(q+p+1)} \quad \mbox{for $p \geq 4$}. 
\end{align*}
\end{lemma}
\begin{proof} Denote by $\tilde{\Delta}$ and $\tilde{\nabla}$ the Laplace operator and gradient with respect to the standard metric on $S^2$. Then 
\begin{align} \label{basic}
\tilde{\Delta} \tilde{x}^i=-2 \tilde{x}^i \quad \text{ and } \quad \tilde{\nabla}  \tilde{x}^i\cdot \tilde{\nabla} \tilde{x}^j =\delta^{ij}-\tilde{x}^i \tilde{x}^j.
\end{align}
Therefore, 
\begin{align*}
	\tilde{\Delta}  ((\tilde{x}^i)^p (\tilde{x}^j)^q)=& p(p-1) (\tilde{x}^i)^{p-2} (\tilde{x}^j)^q+q(q-1) (\tilde{x}^i)^p (\tilde{x}^j)^{q-2}\\
& -(p+q)(p+q+1) (\tilde{x}^i)^p (\tilde{x}^j)^q.
\end{align*}
Integrating by parts, we obtain the first formula. The remaning equations follow by induction. 
\end{proof}

\begin{proposition} \label{pr:sphere}
Let $A=\tilde{x}^1 (\tilde{x}^2)^3$. Then
\[
	\int_{S^2}  \tilde{x}^3* d\left( dA \left[A\tilde{\Delta}A-|\tilde{\nabla} A|^2\right]\right) d\mu_{S^2}=\frac{-16 \pi}{ 7\cdot 11}.
\]
\end{proposition}
\begin{proof}
In this proof, we denote $\tilde x^i$ simply by $x^i$. Using \eqref{basic}, we compute $\tilde{\Delta} A=-20 A+6x^1 x^2$, $|\tilde{\nabla} A|^2= (x^2)^6+9(x^1)^2 (x^2)^4-16( x^1)^2 (x^2)^6$, and thus $A\tilde{\Delta} A-|\tilde{\nabla} A|^2=-4 A^2-3(x^1)^2(x^2)^4-(x^2)^6$. The term $-4A^2$ does not contribute to the integral. We simplify
\[
\begin{split}
	*d (d A(-3(x^1)^2(x^2)^4-(x^2)^6)) & = (-6(x^2)^8+6(x^1)^2 (x^2)^6)*(dx_1\wedge dx_2) \\
& = x^3(-6(x^2)^8+6(x^1)^2 (x^2)^6).
\end{split}
\]
It suffices to show that $\int_{S^2}(x^3)^2(-6(x^2)^8+6(x^1)^2 (x^2)^6)\,d\mu_{S^2}\not= 0$.  Using $(x^1)^2=1-(x^3)^2-(x^2)^2$, we have
\[
\begin{split}
& \int_{S^2}(x^3)^2(-6(x^2)^8+6(x^1)^2 (x^2)^6)\,d\mu_{S^2}\\
=&6\left(-2\int_{S^2} (x^3)^2 (x^2)^8d\mu_{S^2}+\int_{S^2} (x^3)^2 (x^2)^6d\mu_{S^2}-\int_{S^2} (x^3)^4 (x^2)^6 \,d\mu_{S^2}\right)\\
= &\frac{-16 \pi}{7\cdot 11}.
\end{split}
\] 
Lemma \ref{integral_lemma} is used in the last equality.
\end{proof}

\begin{proof}[Proof of Theorem~\ref{th:hypersurface-angular-momentum}]
Note that the fall-off rate $f = O_2(r^{1/3})$ implies that $E=0$, and thus $|P|=0$ by the positive mass theorem. 

By Proposition~\ref{prop:angular-momentum-sphere}, the angular momentum is finite. We only need to show that it is not zero with respect to one of the rotation vector fields. By Proposition~\ref{prop:angular-momentum-sphere} and Proposition~\ref{pr:sphere}, 
\begin{align*}
	&J (x^1 \frac{\partial}{\partial x^2}-x^2\frac{\partial}{\partial x^1}) \\
	&=-\frac{1}{24\pi}\int_{S^2}  \tilde{x}^3* d\left( dA \left[A{\tilde{\Delta}}A-|{\tilde{\nabla}} A|^2\right]\right) \, d\mu_{S^2}= \frac{2}{3\cdot7\cdot 11}.
\end{align*}

Furthermore, we can see that the center of mass is zero with respect to the coordinate chart $\{x\}$. In fact, it is easy to see that in  the coordinate chart $\{x\}$, the coordinate spheres are symmetric with respect to reflection through the origin since the function $f$ is even. 
\end{proof}
\begin{remark}
Note that hypersurfaces of the form $t=r^{\frac{1}{3}}A(u^a)$ do not satisfy the Regge-Teitelboim condition unless $A$ is an odd function. If $A$ is odd, by Proposition~\ref{prop:angular-momentum-sphere}, the angular momentum is zero.
\end{remark}
\begin{remark}
It is worthwhile to remark that the total angular momentum integral \eqref{eq:angular-momentum-def} may be computed with respect to the induced metric on the spheres $\{|x| = r\}$. Under the Regge-Teitelboim asymptotics, the limiting value is the same. However, our example in Theorem~\ref{th:hypersurface-angular-momentum} does not satisfy the Regge-Teitelboim assumption, and the limiting value may differ by a finite vector when the integral is computed with respect to the induced metric. Nevertheless, the angular momentum with respect to the induced metric still can be non-zero.\end{remark}

\subsection{Non-zero center of mass}
We construct a hypersurface in the Minkowski space with zero energy, linear momentum, and angular momentum, but its center of mass integral is not zero.
\begin{theorem}\label{th:hypersurface-center-mass}
Suppose that $f = A(u^a) $ where $A = \tilde x^1 +  \tilde{x}^1\tilde{x}^2$. Then the hypersurface $t = f$ in the Minkowski spacetime satisfies
\[	
	E= 0, \quad |P| = 0, \quad \mbox{and} \quad |J|=0. 
\] 
The components of the center of mass integral $C^\alpha $ are 
\[
	C^1=0, \quad C^2 =\frac{-1}{5}, \quad C^3=0.
\]
\end{theorem}
To prove the theorem, we need the following computational result for the center of mass integral.
\begin{proposition} \label{prop:center-mass-sphere}
Let $t=f(r,u^a)= A(u^a)$ be a hypersurface in the Minkowski spacetime for some function $A \in C^2 (S^2)$. Then the center of mass integral is finite and 
\begin{equation} \label{eq:center-mass-sphere}
	C^{i}  =\frac{-1}{8 \pi} \int_{S^2} |\tilde{\nabla} A|^2 \tilde x^{i}  \, d\mu_{S^2}.
\end{equation}
\end{proposition}
\begin{proof}
Note that the induced metric on the graph is $g_{ij} = \delta_{ij} -f_if_j$. The integrand of the center of mass integral \eqref{de:center-of-mass} becomes
\begin{align*}
 &x^i\sum_{j,k}\left(\frac{\partial g_{jk}}{\partial x^k}-\frac{\partial g_{kk} }{\partial x^j}\right)\frac{x^j}{|x|}  -\sum_k \left((g_{k i} -\delta_{ki})\frac{x^k}{|x|} -(g_{kk} - \delta_{kk})\frac{x^{i}}{|x|}\right) \\
 &= x^i \sum_{j,k}( -f_{jk} f_k - f_{kk} f_j+ 2f_{kj }f_{k}) \frac{x^j}{|x|} - \left(-\sum_k f_k f_i \frac{x^k}{|x|} + |\nabla f|^2 \frac{x^i}{|x|}\right)\\
 &=  x^i \sum_{j,k}( f_{kj} f_k - f_{kk} f_j)\frac{x^j}{|x|} + \sum_k f_k f_i \frac{x^k}{|x|}  - |\nabla f|^2 \frac{x^i}{|x|}.
\end{align*}
Note that $\sum_j f_jx^j= r \partial_r f$ and $(\sum_j f_j x^j)_{k}= \sum_j f_{kj}x^j + f_k$. Since $f$ is independent of $r$, we have  
\[  
	\sum_i f_ix^i=0  \quad \mbox{and} \quad \sum_j f_{kj}x^j =- f_k.  
\]
Therefore, we obtain
\[
\begin{split}
	C^{i}    = &  \frac{1}{ 16 \pi } \lim_{r\rightarrow \infty} \int_{|x| = r}  \frac{1}{r} \Big( x^i\sum_{ j,k} (f_{kj}f_k - f_{kk} f_j) x^j + \sum_k f_k f_i x^k - |\nabla f|^2 x^i \Big)  \, d \sigma_0\\ 
=&  \frac{1}{ 16 \pi } \lim_{r\rightarrow \infty} \int_{|x| = r}  \frac{1}{r} \Big( x^i\sum_{ j,k} ( f_{kj}f_k) x^j -  |\nabla f|^2 x^i \Big)  \, d \sigma_0\\ 
=& \frac{-1}{ 8 \pi } \lim_{r\rightarrow \infty}   \int_{|x| = r} |\nabla f|^2  \frac{x^{i}}{r}   \, d \sigma_0= \frac{-1}{8 \pi} \int_{S^2} |\tilde{\nabla} A|^2  \tilde x^{i}   \, d\mu_{S^2}.
\end{split}
\]
\end{proof}

\begin{proof}[Proof of Theorem~\ref{th:hypersurface-center-mass}]
From the fall-off rates of $g$ and $\pi$, it is easy to see that $E=0$ and $|P|=0$. By Proposition~\ref{pr:angular-momentum-spherical} and \eqref{eq:cross-term-p-zero}, the angular momentum is zero. 

By Proposition~\ref{prop:center-mass-sphere}, the center of mass integral is finite. We only need to show that it is not zero for some $i = 1, 2, 3$. Let $A=\tilde{x}^1+\tilde{x}^1 \tilde{x}^2   $. By Proposition~\ref{prop:center-mass-sphere} and Equation \eqref{basic}, 
\begin{align*}
	C^2 &= \frac{-1}{8 \pi} \int_{S^2} |\tilde{\nabla} A|^2  \tilde x^2   \, d\mu_{S^2} \\
	&= \frac{-1}{8 \pi}\int_{S^2} |\tilde{\nabla} \tilde x^1 + \tilde{\nabla}(\tilde x^1 \tilde x^2)|^2  \tilde x^2  \, d\mu_{S^2} \\
&=\frac{-1}{4 \pi}     \int_{S^2} \tilde x^2 \tilde{\nabla} \tilde x^1 \cdot  \tilde{\nabla}(\tilde x^1 \tilde x^2) \, d\mu_{S^2} \\
&=\frac{-1}{4 \pi}     \int_{S^2}[ (\tilde x^2)^2 -2 (\tilde x^1)^2(\tilde x^2)^2 ]\, d\mu_{S^2}\\
& =\frac{-1}{5}.
\end{align*}
In the last equality,   we use Lemma \ref{integral_lemma}.

By a similar computation, one shows that $C^1=C^3=0$
\end{proof}
\begin{remark}
The induced metric $g$ and the second fundamental form $k$ for the hypersurface $t=A(u^a)$ is 
\[
g_{ij} =\delta_{ij} + O(r^{-2}),  \quad k_{ij} = O(r^{-2}).
\]
In general, the Regge-Teitelboim condition does not hold on these hypersurfaces unless the function $A$ is odd. However, if $A$ is odd, it is easy to see that the center of mass integral is zero due to parity. 
\end{remark}
\begin{remark}
One can construct  hypersurfaces with finite and non-zero center of mass integral defined by $t=r^kA(u^a)+r^{-k}B(u^a)$ for some $\frac{1}{2} \ge k \ge 0 $. In particular, if $k=\frac{1}{2}$,  $A$ is  odd   and $B$ is even, then the induced metric and second fundamental form satisfy
\begin{align*}
&g_{ij} =\delta_{ij} + O(r^{-1}),  \quad k_{ij} = O(r^{-\frac{3}{2}})\\
&g^{\textup{odd}}_{ij} = O(r^{-2}),  \quad \;\;\quad k^{\textup{even}}_{ij} = O(r^{-\frac{5}{2}}).
\end{align*}
The fall-off rates of $k$ and $k^{\textup{even}}$ are on the borderline cases for the Regge-Teitelboim condition.
\end{remark}
\section{Hypersurfaces in the Schwarzschild spacetime} \label{sec:Schwarzschild}
In this section, our goal is to find spacelike hypersurfaces in the Schwarzschild spacetime whose angular momentum is not zero. The Schwarzschild spacetime metric of mass $m$ outside the event horizon is given by 
\[
	ds^2 =- \left( 1-\frac{2m}{r}\right) dt^2 + \left( 1 - \frac{2m}{r} \right)^{-1}dr^2 + r^2 \tilde{\sigma}_{ab} du^a du^b,
\]
where $r > 2m$, $\{u^a\}, a = 1,2,$ are spherical coordinates on the unit sphere, and the $\tilde{\sigma}_{ab}$ is the metric on the unit sphere. The exterior of a spacelike hypersurface in the Schwarzschild space  is the graph of $t= f(r, u^a)$ for $r>2m$. Let $\bar{g}$ be the metric  on the slice $\{ t=0\}$ and $\bar{\nabla}$ be the covariant derivative of $\bar{g}$.
\begin{theorem}\label{th:Sch-hypersurface-angular-momentum}
The exterior of a spacelike hypersurface in the Schwarzschild spacetime of mass $m$ can be expressed as  the graph of $t = f(r, u^a), r > 2m$. Suppose $f = r^{\frac{1}{3}} A(u^a) $ where $A = \tilde{x}^1(\tilde{x}^2)^3$. Then the spacelike hypersurface is asymptotically flat and satisfies
\[	
	E= m, \quad P = 0, \quad \mbox{and} \quad C=0.
\] 
The angular momentum $J(Y)$ is finite for any $Y = \frac{\partial}{\partial x^i} \times \vec{x}, i = 1, 2, 3$ and  
\[
	J (x^1 \frac{\partial}{\partial x^2}-x^2\frac{\partial}{\partial x^1}) =\frac{2}{3\cdot7\cdot 11}.
\]
\end{theorem}
\begin{remark}
Rescaling $f$ by a constant $\lambda >0$, we have asymptotically flat manifolds of mass $m$ and of arbitrarily large angular momentum. In particular, this provides an explicit example that the mass-angular momentum inequality $m \ge |J|$ does not hold for the ADM definition~(cf. \cite{Huang-Schoen-Wang:2011}). 
\end{remark}
In order to obtain the momentum tensor and then the angular momentum, we first compute the first and second fundamental forms of a spacelike hypersurface in the Schwarzschild spacetime.

\begin{proposition}\label{pr:schwarzschild-first}
The induced metric of a hypersurface $t = f(r, u^a)$ in the Schwarzschild spacetime of mass $m$ is 
\[
\begin{split}
	g =& \left[ \left( 1 - \frac{2m}{r} \right)^{-1} - \left( 1 - \frac{2m}{r} \right) f_r^2 \right] dr^2  - 2\left( 1 - \frac{2m}{r} \right) f_r f_a dr du^a\\
& + \left[r^2 \tilde{\sigma}_{ab} - \left( 1 - \frac{2m}{r} \right)f_a f_b  \right]  du^a du^b.
\end{split}
\] 
Denote by $g^{rr}$, $g^{ra}$, and $g^{ab}$ the coefficients of the inverse metric. The inverse metric is considered as a $(2,0)$ tensor. Then
\begin{align*}
	g^{rr} &= \left(1 - \frac{2m}{r} \right) + \frac{\left( 1 - \frac{2m}{r} \right)^2 }{\left( 1 - \frac{2m}{r} \right)^{-1} - | \bar{\nabla} f|^2} f_r^2\\
g^{ra}& = \frac{\left( 1- \frac{2m}{r} \right) r^{-2} \tilde{\sigma}^{ab} }{ \left( 1-\frac{2m}{r} \right)^{-1} - | \bar{\nabla} f|^2}f_b f_r\\
g^{ab} & =r^{-2} \tilde{\sigma}^{ab} + \frac{r^{-4}\tilde{\sigma}^{ac} \tilde{\sigma}^{bd} }{\left( 1-\frac{2m}{r} \right)^{-1} - | \bar{\nabla} f|^2}f_c f_d.
\end{align*}
The timelike unit normal vector is 
\[
	\nu = \frac{1}{\sqrt{ \left(1 - \frac{2m}{r} \right)^{-1} - |\bar{\nabla} f|^2}} \left[ \left( 1 - \frac{2m}{r} \right)^{-1} \partial_t + \left( 1 - \frac{2m}{r} \right) f_r \partial_r + r^{-2} \tilde{\sigma}^{ab} f_a \partial_b \right].
\]
\end{proposition}
\begin{proof}
For any coordinate system $\{x\}$ on the $\{t=0\}$-slice, the induced metric on the hypersurface $t = f(x)$ is 
\[
	g_{ij} = \bar{g}_{ij} + f_i f_j g_{tt}= \bar{g}_{ij} - f_i f_j  \left( 1 - \frac{2m}{r} \right)
\]
and the inverse metric is 
\[
	g^{ij} = \bar{g}^{ij} + \frac{ \bar{g}^{ik} \bar{g}^{jl} f_k f_l} {\left( 1 - \frac{2m}{r} \right)^{-1} - |\bar{\nabla} f|^2}.
\]
Using the spherical coordinates $\{r, u^a\}$ on the $0$-slice and noting
\[
	\bar{g}^{-1} = \begin{pmatrix} 1- \frac{2m}{r} & 0 \\ 0 & r^{-2} \tilde{\sigma}^{ab} \\ \end{pmatrix},
\]
we obtain the first two equations. 

Let $e_r = \partial_r + f_r \partial_t$ and $e_a = \partial_a + f_a \partial_t, a = 1, 2 $ be a basis of the tangent space of the graph. It is straightforward to check that the unit timelike vector $\nu$ is normal to $e_r$ and $e_a$.
\end{proof}

\begin{proposition}\label{pr:schwarzschild-second}
Let $k$ be the second fundamental form of $t= f(r, u^a)$ in the Schwarzschild spacetime of mass $m$. Let  
\[
	w=\sqrt{\left( 1 - \frac{2m}{r}\right)^{-1} - |\bar{\nabla} f|^2}.
\] Then $k = k_{rr} dr^2 + 2 k_{ra} dr du^a + k_{ab} du^a du^b$ where
\begin{align*}
	k_{rr} &= \frac{1}{w }  \left(   f_{rr} +3 f_{r} \frac{m}{r^2} \left( 1 - \frac{2m}{r} \right)^{-1} -f_r^3 \frac{m}{r^2} \left( 1- \frac{2m}{r} \right) \right)\\
	k_{ra} & = \frac{1}{w} \left( f_{ra} -  \frac{f_a}{r}+ f_a \frac{m}{r^2} \left( 1 - \frac{2m}{r} \right)^{-1} - f_r^2 f_a \frac{m}{r^2} \left( 1 - \frac{2m}{r} \right)\right)\\
	k_{ab} & = \frac{1}{w} \left( f_{ab} + (r\tilde{\sigma}_{ab} - f_a f_b \frac{m}{r^2} ) f_r \left( 1 - \frac{2m}{r} \right) - \Gamma_{ab}^c f_c \right).
\end{align*}
\end{proposition}
\begin{proof}
Let $e_r = \partial_r + f_r \partial_t$ and $e_a = \partial_a + f_a \partial_t, a = 1, 2 $ be tangent vectors of the graph. Recall $k_{ij} = - ds^2 (\nu, \nabla^{ds^2}_{e_i} e_j)$. The Christoffel symbols of $ds^2$ are 
\begin{align*}
&\Gamma_{rr}^r = -\left( 1 - \frac{2m}{r} \right)^{-1} \frac{m}{r^2}\\
&\Gamma_{ra}^b = \frac{1}{r} \delta_a^b\\
& \Gamma_{ab}^r = - \left( 1 - \frac{2m}{r} \right) \tilde{\sigma}_{ab} r\\
& \Gamma_{rt}^t = \frac{m}{r^2} \left( 1 - \frac{2m}{r} \right)^{-1}\\
&\Gamma_{tt}^r = \frac{m}{r^2} \left( 1 - \frac{2m}{r} \right),
\end{align*} $ \Gamma_{ab}^c$ are the same as those on $S^2$, 
and all other Christoffel symbols are zero. The desired result follows from direct computations.
\end{proof}

\begin{proof}[Proof of Theorem~\ref{th:Sch-hypersurface-angular-momentum} ]
It follows from the fall-off rate $p=\frac{1}{3}$ that $E=m$. To obtain the angular momentum, we compute $\pi_{ra} = k_{ra} - (\mbox{tr}_g k)g_{ra} $. By Proposition~\ref{pr:schwarzschild-first} and Proposition~\ref{pr:schwarzschild-second} and letting $f (r, u^a)  = r^{\frac{1}{3}} A(u^a)$,  we check that
\begin{align*}
	\pi_{ra} &= k_{ra} - (\mbox{tr}_g k) g_{ra} \\
	&= \frac{1}{w} \left[ f_{ra} - \frac{f_a}{r} + f_a \frac{m}{r^2} + (f_{rr} + r^{-2} \tilde{\Delta} f + 2r^{-1} f_r )f_rf_a\right] + o(r^{3p-3}).
\end{align*}
where
\begin{align*}
	\frac{1}{w} = &\frac{1}{ \sqrt{\left( 1 - \frac{2m}{r}\right)^{-1} - |\bar{\nabla} f|^2}}\\
= & 1 - \frac{m}{r} + \frac{1}{2} (p^2 A + |\tilde{\nabla} A|^2 ) r^{2p-2} + o(r^{2p-2}).
\end{align*}

Therefore, comparing with \eqref{eq:cross-term-f} of the hypersurface in the Minkowski spacetime, the only extra term involving $m$ that would contribute to the angular momentum integral is
\begin{align*}
	-\frac{m}{r} \left(f_{ra} -2 \frac{f_a}{r} \right) = -(p-2)m A_a r^{p-2}.
\end{align*}
Since the above term is a closed form on $S^2$, its contribution to the angular momentum is zero by Proposition~\ref{pr:angular-momentum-spherical}. As computed in Proposition~\ref{prop:angular-momentum-sphere} and by Proposition~\ref{pr:sphere}, we conclude that
\begin{align*}
	&J (x^1 \frac{\partial}{\partial x^2}-x^2\frac{\partial}{\partial x^1}) \\
	&=-\frac{1}{24\pi}\int_{S^2}  \tilde{x}^3* d\left( dA \left[A{\tilde{\Delta}}A-|{\tilde{\nabla}} A|^2\right]\right) \, d\mu_{S^2}= \frac{2}{3\cdot 7\cdot 11}.
\end{align*}

For the linear momentum, we apply the invariance of mass for asymptotically flat spacetime by Chru{\'s}ciel in \cite{Chrusciel-88}. Since the hypersurface and the static slice are both asymptotically flat of order $\frac{2}{3}$ and they differ by a supertranslation of order $O_2(r^{\frac{1}{3}})$, the ADM mass and linear momentum of the hypersurface are the same as those of the static slice. Hence the linear momentum of the hypersurface is zero. 

For the center of mass, since the hypersurface is symmetric with respect to the origin of the coordinate system $\{ x^i\}$, its center of mass is zero.
\end{proof}

\begin{ack}
P.-N. Chen was partially supported by NSF DMS-1308164. L.-H.~Huang was partially supported by NSF DMS-1308837. M.-T. Wang was partially supported by NSF  DMS-1105483 and DMS-1405152. S.-T. Yau was partially supported by NSF through DMS-0804454 and PHY-0937443. This material is based upon work partially supported by NSF under Grant No. 0932078 000, while the second author was in residence at the Mathematical Sciences Research Institute in Berkeley, California, during the Fall 2013 program in Mathematical General Relativity. We also  thank the referee for very careful reading and for pointing out some missing factors in the computations of the first version of the paper. 
\end{ack}

\bibliographystyle{amsplain}
\bibliography{angular_momentum_references}
\end{document}